\documentclass{publicadapt}

\usepackage{amssymb}%
\usepackage{amsmath}%
\usepackage{amsthm}%
\usepackage{bbm}
\usepackage{tikz-cd}
\usepackage[cal=boondox,scr=boondoxo]{mathalfa}
\usepackage{genyoungtabtikz}
\usepackage{ytableau}
\usepackage{xcolor}%

\allowdisplaybreaks%

{\theoremstyle{plain}%
 \newtheorem{theorem}{Theorem}

}
{\theoremstyle{remark}

}
{\theoremstyle{definition}
\newtheorem{definition}{Definition}
\newtheorem{example}{Example}
}

 \title{Semisimple algebras related to immaculate tableaux}

 \author{John M.\ Campbell}
 \address{Department of Mathematics and Statistics, Dalhousie University, Halifax, Nova Scotia, Canada}
 \email{jh241966@dal.ca}
 \thanks{The author was supported by a Killam Postdoctoral Fellowship from the Killam Trusts.}

 \date{May, 05, 2025}


\begin{document}

\begin{abstract}
 Given a direct sum $A$ of full matrix algebras, if there is a combinatorial interpretation associated with both the dimension of $A$ and 
 the dimensions of the irreducible $A$-modules, then this can be thought of as providing an analogue of the famous Frobenius--Young 
 identity $n! = \sum_{\lambda \vdash n} ( f^{\lambda} )^{2}$ derived from the semisimple structure of the symmetric group algebra 
 $ \mathbb{C}S_{n}$, letting $f^{\lambda}$ denote the number of Young tableaux of partition shape $\lambda \vdash n$. By letting 
 $g^{\alpha}$ denote the number of standard immaculate tableaux of composition shape $\alpha \vDash n$, we construct an algebra $ 
 \mathbb{C}\mathcal{I}_{n}$ with a semisimple structure such that $\dim \mathbb{C}\mathcal{I}_{n} = \sum_{\alpha \vDash n} ( 
 g^{\alpha} )^{2}$ and such that $\mathbb{C}\mathcal{I}_{n} $ contains an isomorphic copy of $\mathbb{C}S_{n}$. We bijectively 
 prove a recurrence for $\dim \mathbb{C}\mathcal{I}_{n}$ so as to construct a basis of $\mathbb{C}\mathcal{I}_{n}$ indexed by 
 permutation-like objects that we refer to as \emph{immacutations}. We form a basis $\mathcal{B}_{n}$ of $\mathbb{C}\mathcal{I}_{n}$ 
 such that $\mathbb{C} \mathcal{B}_n$ has the structure of a monoid algebra in such a way so that $\mathcal{B}_n$
 is closed under the  multiplicative operation of $\mathbb{C} \mathcal{I}_n$,  yielding a monoid structure on the
 set of order-$n$ immacutations.  
\end{abstract}

 \maketitle

\section{Introduction}
 The representation theory of the symmetric group $S_{n}$ is of basic importance in many areas of mathematics and may be seen as 
 providing a foundation for a rich interplay among combinatorial and algebraic objects. Young's construction of matrix units for the 
 irreducible modules of the group algebra $\mathbb{C} S_{n}$ gives rise to one of the most famous formulas in the entire discipline of 
 combinatorics, namely, the \emph{Frobenius--Young identity} such that 
\begin{equation}\label{FrobeniusYoung}
 n! = \sum_{\lambda \vdash n} ( f^{\lambda} )^{2} 
\end{equation}
 for a positive integer $n$ and for integer partitions $\lambda$ of $n$ and for the number $f^{\lambda}$ of standard Young tableaux of 
 shape $\lambda$, referring to Section \ref{sectionbackground} for preliminaries and for background material. The importance of the 
 identity in \eqref{FrobeniusYoung} within both combinatorics and representation theory leads us to consider extensions and variants 
 of \eqref{FrobeniusYoung} from both combinatorial and representation-theoretic perspectives. This has led us to formulate and prove 
 an analogue of the Frobenius--Young identity using a generalization of standard Young tableaux arising from the Pieri rule for the 
 \emph{immaculate basis} introduced in a seminal paper by Berg et al.\ \cite{BergBergeronSaliolaSerranoZabrocki2014}. 

 By the semisimplicity of the group algebra $\mathbb{C} S_{n}$, and by Frobenius' Fundamental Theorem of Representation Theory 
 \cite[\S1]{GarsiaEgecioglu2020} or an equivalent version of 
 Maschke's Theorem, 
 there exists a basis of $\mathbb{C} S_{n}$ satisfying the matrix unit multiplication rules (reviewed in 
 Section \ref{sectionbackground} below). The irreducible $\mathbb{C} S_{n}$-modules are indexed by partitions $\lambda \vdash n$, 
 and, for such a partition $\lambda$, every irreducible $\mathbb{C} S_{n}$-module is of dimension $f^{\lambda}$ and multiplicity 
 $f^{\lambda}$. So, by taking all of the $(f^{\lambda})^{2}$ matrix units among the irreducible $\mathbb{C} S_{n}$-modules 
 corresponding to a given partition $\lambda$ of $n$, and by then taking the matrix units corresponding to the irreducible modules that 
 remain, we obtain a basis of $\mathbb{C}S_{n}$ of size $ \sum_{\lambda \vdash n} ( f^{\lambda} )^{2}$, or, equivalently, of size $n!$. 
 In an equivalent way, by letting $ \mathcal{M}_{m}(\mathbb{C})$ denote the algebra consisting of $m \times m$ matrices with 
 complex entries, the identity in \eqref{FrobeniusYoung} may be obtained by taking the dimension of each side of 
 the isomorphic equivalence 
\begin{equation}\label{decomposeCSn}
 \mathbb{C} S_{n} \cong \bigoplus_{\lambda \vdash n} \mathcal{M}_{f^{\lambda}}(\mathbb{C}), 
\end{equation}
 with the decomposition of $ \mathbb{C} S_{n}$
 in \eqref{decomposeCSn} being referred to as the \emph{Wedderburn decomposition}
 of $ \mathbb{C} S_{n}$. 
 The rich history associated with the bijective proof of \eqref{FrobeniusYoung} via the RSK correspondence motivates the development 
 of research topics related to extensions or variants of \eqref{FrobeniusYoung}. In this direction, since noncommutative Schur-like 
 bases give rise, via Pieri rules, to composition tableau generalizations of Young tableaux, this raises questions as to what would be 
 appropriate as an analogue of \eqref{FrobeniusYoung} involving composition tableaux and how such an analogue, in turn, could 
 extend the semisimple structure indicated in \eqref{decomposeCSn}. 

 Given a family of combinatorial objects related to Schur functions, such as the tableaux arising in product rules associated with the 
 Schur basis, constructing analogues of such combinatorial objects with the use of noncommutative versions of Schur functions is 
 often motivated by applications related to Schur positivity and to representation-theoretic problems. Informally, in view of how the 
 Hopf algebra $\textsf{NSym}$ projects onto $\textsf{Sym}$, Schur-like bases of the former Hopf algebra can be thought of as 
 ``containing more information'' in a way that could, ideally, be useful in terms of developing a deeper understanding as to the behavior 
 of commutative Schur functions. In this regard, active research related to the \emph{immaculate basis} of $\textsf{NSym}$ 
 \cite{BergBergeronSaliolaSerranoZabrocki2014} has to do with its close relationship with the Schur basis of $\textsf{Sym}$. 

 Since standard immaculate tableaux provide a natural analogue of standard Young tableaux from both algebraic and combinatorial 
 perspectives, this leads us to introduce an analogue of \eqref{FrobeniusYoung} with the use of standard immaculate tableaux. Again 
 referring to Section \ref{sectionbackground} for background and preliminaries, by letting $\alpha$ denote an integer composition, 
 and by letting $g^{\alpha}$ denote the number of standard immaculate tableaux of shape $\alpha$, a main object of study in this work 
 is the \emph{immaculate algebra} 
\begin{equation}\label{CIndefinition}
 \mathbb{C} \mathcal{I}_{n} := \bigoplus_{\alpha \vDash n} \mathcal{M}_{g^{\alpha}}(\mathbb{C}) 
\end{equation}
 that we introduce. Since $ \mathbb{C} \mathcal{I}_{n}$ is defined by analogy with the direct sum decomposition in 
 \eqref{decomposeCSn}, this raises questions as to how the combinatorics associated with immaculate tableaux could give rise to a 
 product rule for $\mathbb{C} \mathcal{I}_{n}$ that emulates or reflects the operation given by the composition of permutations. 
 Such questions are 
 inspired by past research concerning 
 immaculate tableaux 
 \cite{GaoYang2016,SunHu2018}, 
 immaculate functions 
 \cite{BergBergeronSaliolaSerranoZabrocki2017,Li2018,LoehrNiese2021}, 
 and dual immaculate functions 
 \cite{AllenHallamMason2018,BergBergeronSaliolaSerranoZabrocki2015,BergeronSanchezOrtegaZabrocki2016, 
 Daugherty2024,MasonSearles2021, 
NieseSundaramvanWilligenburgVegaWang2024,
NieseSundaramvanWilligenburgVegaWang2023,
NieseSundaramvanWilligenburgWang2024,NovelliThibonToumazet2018}. 

 Experimentally, using the On-Line Encyclopedia of Integer Sequences \cite{oeis}, we have discovered a recursion for the sequence 
 \begin{equation}\label{dimnumerical}
 ( \dim \mathbb{C} \mathcal{I}_{n} : n \in \mathbb{N} ) = (1, 2, 7, 35, 236, 2037, 21695, 277966, 4198635, 73558135, \ldots), 
\end{equation}
 and we introduce a bijective proof of this recurrence. 
 Our bijective approach gives rise to a combinatorial 
 interpretation for the dimension of $ \mathbb{C} \mathcal{I}_{n} $ in terms of tuples of 
 integer sets that may be seen as analogues of permutations and that we refer to \emph{immacutations} of order $n$. This and the 
 definition in \eqref{CIndefinition} give rise to our analogue 
\begin{equation}\label{immacutationsFY}
 \text{$\#$ of immacutations of order $n$} = \sum_{\alpha \vDash n} (g^{\alpha})^{2} 
\end{equation}
 of the Frobenius--Young identity. 

\section{Background and preliminaries}\label{sectionbackground}
 Since the Frobenius--Young identity and how it can be generalized in relation to noncommutative symmetric functions are central to our 
 work, we begin with a review of both Young tableaux and the algebra $\textsf{Sym}$, as below. 

\subsection{Young tableaux and commutative symmetric functions}
 Much about our terminology and notation concerning symmetric functions is adapted from Macdonald's classic monograph on 
 symmetric functions \cite{Macdonald1995}. In this direction, we set $ \textsf{Sym}^{(n)} = \mathbbm{k}[x_{1}, x_{2}, \ldots, 
 x_{n}]^{S_{n}} $ for indeterminates $x_{1}$, $x_{2}$, $\ldots$, $x_{n}$ and for a field $\mathbbm{k}$ that we set as $\mathbb{Q}$ for 
 convenience, and where $S_{n}$ acts on the polynomial ring $\mathbb{Q}[x_{1}, x_{2}, \ldots, x_{n}]$ 
 by permuting the given variables (cf.\ \cite[\S1.2]{Macdonald1995}). 
 This leads us to a graded ring decomposition, writing
 $$ \textsf{Sym}^{(n)} = \bigoplus_{k \geq 0} \textsf{Sym}^{(n)}_{k}, $$
 letting $ \textsf{Sym}^{(n)}_{k} $ consist of the zero polynomial and homogeneous symmetric polynomials
 of degree $k$. 
 We then construct an inverse limit 
\begin{equation}\label{Symsub}
 \textsf{Sym}_{k} := \lim_{\substack{\longleftarrow\\ n }} \textsf{Sym}_{k}^{(n)}, 
\end{equation}
 and we invite the interested reader to refer 
 to Macdonald's text for details \cite[pp.\ 17--19]{Macdonald1995}. The inverse limit in 
 \eqref{Symsub} leads us to set
 $$ \textsf{Sym} := \bigoplus_{k \geq 0} \textsf{Sym}_{k}. $$
 Define the $r^{\text{th}}$ \emph{elementary symmetric generator} $e_{r}$ 
 so that $e_{0} = 1$ and so that 
 $$ e_{r} = \sum_{i_{1} < i_{2} < \cdots < i_{r}} x_{i_{1}} x_{i_{2}} \cdots x_{i_{r}}. $$
 We then define the \emph{complete homogeneous generator} $h_{m}$ for $m \in \mathbb{N}_{0}$ 
 recursively so that $h_{0} = 1$ and so that 
 $ \sum_{r=0}^{n} (-1)^r e_{r} h_{n-r} = 0$. 
 This leads us to equivalently define $\textsf{Sym}$ so that 
\begin{equation}\label{Symh1h2}
 \textsf{Sym} = \mathbb{Q}[h_{1}, h_{2}, \ldots]. 
\end{equation}
 By setting $\text{deg} \, h_{n} = n$ for positive integers $n$, we thus have that $\textsf{Sym}$ is the free commutative 
 $\mathbb{Q}$-algebra with one generator $h_{n}$ in each degree $n$. 

 An \emph{integer partition} $\lambda$ is a finite tuple of positive integers in weakly decreasing order, and we write $\ell(\lambda)$ in 
 place of the length of this tuple, and write $\lambda = (\lambda_{1}, \lambda_{2}, \ldots, \lambda_{\ell(\lambda)})$. We may identify an 
 integer partition $\lambda$ with a tableau with $\ell(\lambda)$ rows consisting of $\lambda_{i}$ cells in its $i^{\text{th}}$ row for $i 
 \in \{ 1, 2, \ldots, \ell(\lambda) \}$. For a nonempty partition $\lambda$, we write $h_{\lambda} = h_{\lambda_{1}} 
 h_{\lambda_{2}} \cdots h_{\lambda_{\ell(\lambda)}}$, and, for the unique empty partition $()$, we write $h_{()} = h_{0}$. 
 We let $|\lambda|$ denote the sum of the entries of a nonempty partition
 $\lambda$ and adopt the convention 
 whereby the tuple $()$ is such that $|()| = 0$. 
 If $|\lambda| = n$, then we may let this be denoted by writing $\lambda \vdash n$. 
 By letting $ 
 \mathcal{P}$ denote the set of all integer partitions, we thus obtain that the family $\{ h_{\lambda} \}_{\lambda \in \mathcal{P}}$ is a basis 
 of $\textsf{Sym}$, i.e., the \emph{complete homogeneous basis} of $\textsf{Sym}$. The Schur basis 
 $\{ s_{\lambda} \}_{\lambda \in \mathcal{P}}$
 of $\textsf{Sym}$ is typically regarded as \emph{the} basis of $\textsf{Sym}$ 
 and may be defined according to the classical rule
\begin{equation}\label{htosrule}
 h_{\lambda} = \sum_{\mu} K_{\mu, \lambda} s_{\mu}, 
\end{equation}
 for \emph{Kostka coefficients} $K_{\mu, \lambda}$ defined below. 

 The Kostka coefficient $K_{\mu, \lambda}$ is equal to the number of semistandard Young tableaux of weight $\lambda$ and shape 
 $\mu$, where a \emph{semistandard tableau} is a tableau of partition shape labeled with positive integers and with strictly increasing 
 columns and with weakly increasing rows. Similarly, a \emph{standard Young tableau} is a tableau of partition shape such that the 
 labels of this tableau are distinct and consecutive positive integers starting with $1$ and such that the rows and columns are strictly 
 increasing. We henceforth make use of French notation for denoting tableaux, so that the first row of a tableaux is depicted as 
 being the bottommost row. 

\begin{example}
 The standard Young tableaux of size $3$ are as below. 
\begin{equation}\label{sYt3}
 \Yvcentermath1 \Yboxdim{16pt} 
 \young(3,2,1) \ \ \ \ \ \ \ \ \ \ \ \ \ \ 
 \Yvcentermath1 \Yboxdim{16pt} 
 \young(3,12) \ \ \ \ \ \ \ \ \ \ \ \ \ \ 
 \Yvcentermath1 \Yboxdim{16pt} 
 \young(2,13) \ \ \ \ \ \ \ \ \ \ \ \ \ \ 
 \Yvcentermath1 \Yboxdim{16pt} 
 \young(123) 
 \end{equation} 
 The tableaux displayed in \eqref{sYt3} 
 illustrate the special case of the Frobenius--Young identity whereby 
 $$ 3! = \Big( \Yvcentermath1 \Yboxdim{6.5pt} f^{\young(\null,\null,\null)} \Big)^{2} 
 + \Big( \Yvcentermath1 \Yboxdim{6.5pt} f^{\young(\null,\null\null)} \Big)^{2} 
 + \Big( \Yvcentermath1 \Yboxdim{6.5pt} f^{\young(\null\null\null)} \Big)^{2}. $$
\end{example}

\subsection{Semisimple algebras and Young's construction}
 Let $A$ be a finite-dimensional 
 $\mathbbm{k}$-algebra for a field $\mathbbm{k}$. Suppose there exists a family 
 consisting of nonzero elements $ e_{i, j}^{\lambda} \in A$, for $\lambda$
 in some set $I_{1}$, where, for $\lambda \in I_{1}$, 
 the expression 
 $ e_{i, j}^{\lambda}$ is defined for $i$ and $j$ in some index set $I_{2}(\lambda) = I_{2}$, 
 and where such expressions (whenever defined) satisfy 
\begin{equation}\label{matrixunitrules}
 e_{i_{1}, i_{2}}^{\lambda} e_{i_{3}, i_{4}}^{\mu} 
 = \begin{cases} 
 e_{i_{1}, i_{4}}^{\lambda} & \text{if $i_{2} = i_{3}$ and $\lambda = \mu$, and} \\
 0 & \text{otherwise.} 
 \end{cases} 
\end{equation}
 We refer to the multiplication rules in \eqref{matrixunitrules} as the \emph{matrix unit multiplication rules}, since, if we let $E_{i, 
 j}^{(n)}$ for $i, j \in \{ 1, 2, \ldots, n \}$ denote the $n \times n$ matrix with the value of $1$ in its $(i, j)$-entry and with $0$-values 
 everywhere else, then we find that $E_{i_{1}, i_{2}}^{(n)} E_{i_{3}, i_{4}}^{(n)} = E_{i_{1}, i_{4}}^{(n)} $ if $i_{2} = i_{3}$ and that $E_{i_{1}, 
 i_{2}}^{(n)} E_{i_{3}, i_{4}}^{(n)} $ equals the $n \times n$ zero matrix, otherwise, with matrices of this form being referred to as 
 \emph{matrix units}. The orthogonality relations in \eqref{matrixunitrules} for the case whereby the superscripts are unequal may 
 be thought of as being given by the case whereby matrix units are taken from different components in a direct sum of matrix algebras. 

 Given a basis of $A$ or for an $A$-module under the action of left- or right-multiplication by elements of $A$ satisfying the matrix 
 unit multiplication rules, we refer to the elements in such a basis as \emph{matrix units}, and this is consistent with Young's terminology, 
 with regard what is referred to as \emph{Young's construction} \cite[\S1]{GarsiaEgecioglu2020}. We refer to any basis of $A$ consisting 
 of elements $e_{i, j}^{\lambda}$ satisfying \eqref{matrixunitrules} as a \emph{matrix unit basis}. 

 Again with reference to the work of Garsia and E\u gecio\u glu \cite[\S1]{GarsiaEgecioglu2020}, if $A$ has a matrix unit basis, then, 
 according to our notation associated with \eqref{matrixunitrules}, we obtain the direct sum decomposition 
\begin{equation}\label{Adecomposition}
 A \cong \bigoplus_{\lambda \in I_{1}} \mathcal{M}_{\left| I_{2}(\lambda) \right|}(\mathbbm{k}). 
\end{equation}
 Algebras that are isomorphic to (finite) direct sums of full matrix algebras are said to be \emph{semisimple}. The equivalence in 
 \eqref{Adecomposition} gives us that 
\begin{equation}\label{dimAindexsets}
 \dim A = \sum_{\lambda \in I_{1}} \left| I_{2}(\lambda) \right|^{2}, 
\end{equation}
 and Young's construction of a matrix unit basis for symmetric group algebras gives us, from the identity in \eqref{dimAindexsets}, a 
 proof of the Frobenius--Young identity. 

 Young introduced an explicit construction of matrix units for all symmetric group algebras, and we refer to Garsia and E\u gecio\u glu's 
 exposition on this matrix unit construction \cite[\S1]{GarsiaEgecioglu2020}. Omitting details, for a tableau $T$ labeled with consecutive 
 integers starting with $1$ (i.e., an \emph{injective tableau}) of a given shape $\lambda \vdash n$, we define $$ P(T) = \sum_{\alpha \in 
 R(T)} \alpha \ \ \ \text{and} \ \ \ N(T) = \sum_{\beta \in C(T)} \text{sign}(\beta) \beta $$ for the row and column groups $R(T)$ and 
 $C(T)$ associated with $T$ (consisting of permutations that leave the labels of the rows and columns, respectively, of $T$ invariant up to 
 reorderings). We also let 
 $\sigma_{T_{1}, T_{2}}$ denote the unique permutation such that 
 $T_{1} = \sigma_{T_{1}, T_{2}} T_{2}$ for injective tableaux $T_{1}$ and $T_{2}$. 
 By then letting 
 $S_{1}^{\lambda} <_{YFLO} S_{2}^{\lambda} <_{YFLO} \cdots 
 <_{YFLO} S_{f^{\lambda}}^{\lambda} $ 
 denote the standard Young tableaux of shape $\lambda$
 ordered according to Young's First Letter Order (whereby 
 one injective tableau precedes another if the first entry of disagreement is lower), 
 and by setting 
$ \gamma_{i}^{\lambda} = \frac{f^{\lambda}}{n!} N\big( S_{i}^{\lambda} \big) 
 P \big( S_{i}^{\lambda} \big)$, 
 and by writing $\sigma_{S_{i}^{\lambda}, S_{j}^{\lambda}} = \sigma_{i, j}^{\lambda}$, the set of all expressions of the form
\begin{equation*}
 e_{i, j}^{\lambda} = \sigma_{i, j}^{\lambda} \gamma_{j}^{\lambda} 
 \left( 1 - \gamma_{j+1}^{\lambda} \right) \left( 1 - \gamma_{j+2}^{\lambda} \right) 
 \cdots \left( 1 - \gamma_{f^{\lambda}}^{\lambda} \right). 
\end{equation*}
 is a matrix unit basis for $\mathbb{C}S_n$. 

\subsection{Immaculate functions and immaculate tableaux}
 By analogy with \eqref{Symh1h2}, we write 
\begin{equation}\label{NSymH1H2}
 \textsf{NSym} := \mathbb{Q} \langle H_{1}, H_{2}, \ldots \rangle 
\end{equation}
 to define the $\mathbb{Q}$-algebra version of the \emph{algebra of noncommutative symmetric functions}, as introduced in a seminal 
 paper by Gelfand et al.\ \cite{GelfandKrobLascouxLeclercRetakhThibon1995}. By setting the degree of the generator $H_{n}$ as $n$ for 
 positive integers $n$, the definition in \eqref{NSymH1H2} gives us that $\textsf{NSym}$ is the free $\mathbb{Q}$-algebra is one 
 generator in each degree. An integer composition $\alpha$ is a finite tuple of positive integers, and we let $\ell(\alpha)$ denote the 
 number of entries of $\alpha$, writing $\alpha = (\alpha_{1}, \alpha_{2}, \ldots, \alpha_{\ell(\alpha)})$. For a nonempty composition $ 
 \alpha$, we write $|\alpha|$ in place of the sum of the entries of $\alpha$, and, if $|\alpha| = n$ for a positive integer $n$, we may 
 write $\alpha \vDash n$. We adopt the conventions whereby 
 the empty tuple $()$ is the unique composition of $0$. From the definition in \eqref{NSymH1H2}, we find that the 
 bases of $\textsf{NSym}$ are indexed by compositions in a natural way, writing $H_{\alpha} = H_{\alpha_{1}} H_{\alpha_{2}} \cdots 
 H_{\alpha_{\ell(\alpha)}}$ for a nonempty composition $\alpha$, and writing $H_{()} = H_{0} = 1$. This gives rise to the 
 \emph{complete homogeneous basis} $\{ H_{\alpha} \}_{\alpha \in \mathcal{C}}$ of $\textsf{NSym}$, for the set $\mathcal{C}$ of all 
 integer compositions. A main object of study in this paper is given by a family of combinatorial objects 
 (as defined below) arising from an analogue for $ 
 \textsf{NSym}$ of commutative Schur functions. 

\begin{definition}\label{definitionBergtableau}
 (Berg et al.) For compositions $\alpha$ and $\beta$, an \emph{immaculate tableau} 
 of shape $\alpha$ and content $\beta$ is a tableau of the specified composition shape 
 labelled with positive integers so that the following axioms hold: 
\begin{enumerate}

\item The number of cells labeled with $i$ is $\beta_{i}$; 

\item Each row is weakly increasing (when read from left to right); and 

\item The first column is strictly increasing (when read from top to bottom). 

\end{enumerate}

\noindent Moreover, an immaculate tableau is \emph{standard} if it is of content
 $1^{|\alpha|}$ \cite{BergBergeronSaliolaSerranoZabrocki2014}. 
\end{definition}

 By writing $\mathcal{K}_{\alpha, \beta}$ in place of the number of immaculate tableaux
 of shape $\alpha$ and content $\beta$, the immaculate function 
 $\mathfrak{S}_{\alpha}$ may be defined 
 via the expansion rule such that
\begin{equation}\label{Htoimmaculate}
 H_{\beta} = \sum_{\alpha \geq_{\ell} \beta} \mathcal{K}_{\alpha, \beta} \mathfrak{S}_{\alpha}, 
\end{equation}
 by analogy with \eqref{htosrule}, letting 
 the lexicographic ordering on words and compositions be denoted with $\leq_{\ell}$. 
 This allows us to define the \emph{immaculate basis} $\{ \mathfrak{S}_{\alpha} \}_{\alpha \in \mathcal{C}}$ 
 of $\textsf{NSym}$ \cite{BergBergeronSaliolaSerranoZabrocki2014}. 

\begin{example}
 Inputting 
\begin{verbatim}
from sage.combinat.ncsf_qsym.combinatorics import number_of_fCT
number_of_fCT(Composition([1,1,1,1,1]), Composition([2,1,2]))
\end{verbatim}
 into {\tt SageMath}, we find that the number 
 $\mathcal{K}_{(2, 1, 2), (1, 1, 1, 1, 1)}$ of standard immaculate tableaux of shape 
 $(2, 1, 2)$ is $4$, and this is illustrated below. 
\begin{equation*}
 \Yvcentermath1 \Yboxdim{16pt} 
 \young(35,2,14) \ \ \ \ \ \ \ \ \ \ \ \ \ \ 
 \Yvcentermath1 \Yboxdim{16pt} 
 \young(34,2,15) \ \ \ \ \ \ \ \ \ \ \ \ \ \ 
 \Yvcentermath1 \Yboxdim{16pt} 
 \young(45,2,13) \ \ \ \ \ \ \ \ \ \ \ \ \ \ 
 \Yvcentermath1 \Yboxdim{16pt} 
 \young(45,3,12) 
 \end{equation*} 
 Observe that in the second displayed tableau, the labels in the second column are not increasing (from bottom to top). 
\end{example}

\section{Immaculate algebras} 
 Let the integer sequence $(a(n) : n \in \mathbb{N}_{0} )$ be defined recursively so that $a(0) = 1$ and so that 
\begin{equation}\label{anrec}
 a(n) = \sum_{k=0}^{n-1} \binom{n-1}{k}^{2} a(k), 
\end{equation}
 with 
\begin{equation}\label{anumerical}
 (a(n) : n \in \mathbb{N}_{0} ) 
 = (1, 1, 2, 7, 35, 236, 2037, 21695, 277966, 4198635, 73558135, \ldots), 
\end{equation}
 noting the apparent agreement with \eqref{dimnumerical}. The sequence in \eqref{anumerical} is indexed in the the On-Line 
 Encyclopedia of Integer Sequences as {\tt A101514}, and, as noted in this OEIS entry, we have previously conjectured that 
\begin{equation}\label{conjectured}
 \sum_{\alpha \vDash n} ( g^{\alpha} )^{2} = a(n) 
\end{equation}
 holds for all $n \in \mathbb{N}$. We prove \eqref{conjectured} bijectively, as below. 

\begin{theorem}\label{theoremconjectured}
 For all positive integers $n$, the relation in \eqref{conjectured} holds. 
\end{theorem}

\begin{proof}
 The base case for $n = 1$ holds in an immediate way, and we assume that $ \sum_{\alpha \vDash m} ( g^{\alpha} )^{2} = a(m)$ holds 
 for all positive integers $m$ such that $m \leq n$, inductively. So, it remains to prove that this implies that 
\begin{equation*}
 \sum_{\alpha \vDash n + 1} ( g^{\alpha} )^{2} = a(n + 1). 
\end{equation*}
 From the recursive definition of the $a$-sequence in \eqref{anrec}, it remains to prove that our inductive hypothesis implies that 
\begin{equation}\label{wantIHimplies}
 \sum_{k=0}^{n} a(k) \binom{n}{n-k}^{2} = \sum_{\alpha \vDash n + 1} ( g^{\alpha} )^{2}, 
\end{equation}
 noting the use of the symmetry of binomial coefficients. We proceed to define a mapping 
\begin{equation}\label{phicolon}
 \varphi_{n}\colon \mathcal{A}_{n} \to \mathcal{B}_{n}, 
\end{equation}
 as below, where the domain $\mathcal{A}_{n}$ consists of ordered pairs $(T_{1}, T_{2}, S_{1}, S_{2})$ consisting of standard immaculate 
 tableaux $T_{1}$ and $T_{2}$ of the same shape $\beta \vDash k$ (adopting the convention whereby there is a unique ``empty tableau'' 
 $ \boxminus$ with $0$ cells) and subsets $S_{1}$ and $S_{2}$ of $\{ 1, 2, \ldots, n \}$ of the same size $n - k$, and where the codomain 
 $\mathcal{B}_{n}$ consists of ordered pairs $(T_{3}, T_{4})$ of standard immaculate tableaux of the same shape $\alpha \vDash n+1$. 

 As above, let $(T_{1}, T_{2}, S_{1}, S_{2}) \in \mathcal{A}_{n}$. For $i \in \{ 1, 2 \}$, let the 
 $n - k$ elements in $S_{i}$ be ordered by writing 
 $$ S_{i} = \left\{ s_{1}^{(i)} < s_{2}^{(i)} < \cdots < s_{n-k}^{(i)} \right\}. $$
 We also write $t_{j}^{(i)} = s_{j}^{(i)} + 1$ for all $j \in \{ 1, 2, \ldots, n - k \}$, and we write
 $$ \overline{S_{i}} = \left\{ t_{1}^{(i)} < t_{2}^{(i)} < \cdots < t_{n-k}^{(i)} \right\}. $$
 For each of $i \in \{ 1, 2 \}$, we form the tableau 
\begin{equation}\label{displayUi}
 U_{i} = \Yvcentermath1 \Yboxdim{20pt} \young(1{t_{1}^{(i)}}{t_{2}^{(i)}}{\cdots}{t_{n-k}^{(i)}}). 
 \end{equation} 
 As above, we let $T_{1}$ and $T_{2}$ be of the same shape $\beta \vDash k$. For each of $i \in \{ 1, 2 \}$, we form a tableau $V_{i}$ 
 by relabelling $T_{i}$ by replacing the consecutive labels $1$, $2$, $\ldots$, $k$ with, in order, the entries in the tuple 
 obtained by ordering $\{ 2, 3, \ldots, n + 1 \} \setminus \overline{S_{i}}$. We then, again letting $i \in \{ 1, 2 \}$, form a tableau $T_{i + 2}$ by 
 adjoining $U_{i}$ and $V_{i}$, i.e., by placing $V_{i}$ on top of $U_{i}$ to form a new tableau with $n + 1$ labeled cells. 
 We then let \eqref{phicolon} be defined so that 
\begin{equation}\label{phidefineT1T2}
 \varphi_{n}(T_{1}, T_{2}, S_{1}, S_{2}) = (T_{3}, T_{4}). 
\end{equation}
 We see that $\varphi_{n}(T_{1}, T_{2}, S_{1}, S_{2}) $ is indeed a pair of standard immaculate tableaux of the same shape with $n + 
 1$ cells, since, by construction, for each of $i \in \{ 1, 2 \}$, the first column of $T_{i + 2}$ begins with $1$ and then strictly increases 
 according to how the first column of $T_{i}$ is relabelled with (certain) elements in $\{ 2, 3, \ldots, n + 1 \} \setminus \overline{S_{i}}$ 
 in increasing order, and since there are no restrictions on the orderings of the labels in the remaining columns, and since the first row 
 in \eqref{displayUi} is strictly increasing, and since the shifted rows of $V_{i}$ are strictly increasing, and since the number of cells in 
 $T_{i+2}$ is equal to the number $n - k + 1$ of cells in $U_{i}$ plus the number $k$ of cells in $V_{i}$ (or in $T_{i}$). We claim that 
 the mapping defined in \eqref{phidefineT1T2} is bijective. 

 Let $(T_{1}, T_{2}, S_{1}, S_{2})$ and $(T_{1}', T_{2}', S_{1}', S_{2}')$ be distinct elements in the domain in \eqref{phicolon}, and write 
 $$ \varphi_{n}(T_{1}', T_{2}', S_{1}', S_{2}') = (T_{3}', T_{4}'). $$ For each of $i \in \{ 1, 2 \}$, if $S_{i} \neq S_{i}'$, then the first row of 
 $T_{i + 2}$ will not be the same as the first row of $T_{i + 2}'$, according to \eqref{displayUi}. Now, suppose that $T_{i} \neq T_{i}'$, 
 again letting $i \in \{ 1, 2 \}$. If $T_{i}$ and $T_{i}'$ are not of the same shape, then $T_{i+2}$ and $T_{i+2}'$ will not be of 
 the same shape, according to how $T_{i+2}$ is formed by adjoining $U_{i}$ and $V_{i}$, and similarly for $T_{i+2}'$. If $T_{i}$ and 
 $T_{i}'$ are of the same shape but $S_{i} \neq S_{i}'$, then $T_{i+2}$ would not have the same initial row as $T_{i+2}'$, as shown 
 previously. So, it remains to consider the case whereby $T_{i}$ and $T_{i}'$ are of the same shape (but with labels arranged in different 
 ways by assumption that $T_{i} \neq T_{i}'$) and $S_{i} = S_{i}'$. As above, let $T_{i}$ consist of $k$ cells. We thus have that there 
 is a non-identity permutation $\sigma$ of $\{ 1, 2, \ldots, k \}$ such that $T_{i} = \sigma T_{i}'$. Since $S_{i} = S_{i}'$, we find that 
 $U_{i} = U_{i}'$, so that $V_{i}$ is obtained by replacing the consecutive labels $1$, $2$, $\ldots$, $k$ of $T_{i}$ with the consecutive 
 entries of the tuple $\mathcal{T}$ obtained by ordering $\{ 2, 3, \ldots, n + 1 \} \setminus \overline{S_{i}}$, and $V_{i}'$ is obtained by 
 replacing the consecutive labels $1$, $2$, $\ldots$, $k$ of $T_{i}'$ with the consecutive entries of the same tuple $\mathcal{T}$, noting 
 that the same bijective relabelling function $f\colon \{ 1, 2, \ldots, k \} \to \{ 2, 3, \ldots, n + 1 \} \setminus \overline{S_{i}}$ is used in 
 both cases, with $j_{1} < j_{2}$ implying that $f(j_{1}) < f(j_{2})$. This gives us that $V_{i} = \rho V_{i}'$ for a non-identity permutation 
 $ \rho$ of $\{ 2, 3, \ldots, n + 1 \} \setminus \overline{S_{i}}$ such that $\rho = f \circ \sigma$. Since $\rho$ is not the identity 
 permutation, we have that $T_{i}$ and $T_{i}'$ are not the same, giving us the injectivity of $\varphi_{n}$. 

 Let $(T_{3}, T_{4})$ denote a pair of standard immaculate tableaux of the same shape
 with $n + 1$ cells. 
 For $i \in \{ 1, 2 \}$, we take the 
 first row $\mathcal{U}_{i}$ of $T_{i+2}$, remove the $1$-labeled cell, and then 
 take the set $\overline{\mathcal{S}_{i}}$ of the remaining labels, and then shift each label downwards by $1$, 
 yielding a subset $\mathcal{S}_{i}$ of $\{ 1, 2, \ldots, n \}$ 
 of size $n - k$ for some $k \in \{ 0, 1, \ldots, n \}$. 
 Again for $i \in \{ 1, 2 \}$, we let $\mathcal{V}_{i}$ 
 denote the truncated version of $T_{i+2}$ obtained by removing the first row of $T_{i+2}$. 
 This truncated tableau $\mathcal{V}_{i}$ is labeled with the elements in $\{ 2, 3, \ldots, n + 1 \} \setminus 
 \overline{\mathcal{S}_{i}}$, 
 and we relabel $\mathcal{V}_{i}$ by replacing the $j^{\text{th}}$ largest label with $j$, 
 yielding a standard immaculate tableau $T_{i}$. 
 By applying the mapping 
 $ \varphi_{n}$ to $(T_{1}, T_{2}, S_{1}, S_{2})$, our construction gives us that 
 that $\mathcal{S}_{i} = S_{i}$ and $\overline{S}_{i} = S_{i}$ 
 and $\mathcal{U}_{i} = U_{i}$ and 
 $\mathcal{V}_{i} = V_{i}$, so that 
 $ \varphi_{n}$ maps $(T_{1}, T_{2}, S_{1}, S_{2})$
 to $(T_{3}, T_{4})$. 
\end{proof}

\begin{example}
 To illustrate the bijection associated with the case of \eqref{wantIHimplies} whereby 
\begin{equation}\label{a4binomialsum}
 a(4) = a(0) \binom{3}{3}^{2} + a(1) \binom{3}{2}^{2} + 
 a(2) \binom{3}{1}^2 + 
 a(3) \binom{3}{0}^2, 
\end{equation}
 we can think of the initial term $ a(0) \binom{3}{3}^{2} = 1$ as corresponding to 
 the ordered $4$-tuple shown in the correspondence whereby 
\begin{equation}\label{labelhighlight}
 \left( \boxminus, \boxminus, \{ 1, 2, 3 \}, \{ 1, 2, 3 \} \right) \mapsto 
 \Bigg( \raisebox{-0.8ex}{$ \begin{ytableau}
 *(green) 1 & *(white) 2 & *(white) 3 & *(white) 4 
\end{ytableau}, \ 
\begin{ytableau}
 *(green) 1 & *(white) 2 & *(white) 3 & *(white) 4 
\end{ytableau} $} \Bigg), 
\end{equation}
 where the highlighted cells among the tableaux in the tuple shown on the right of \eqref{labelhighlight} are meant to 
 illustrate the definition of \eqref{phicolon}, 
 in the sense that the members of $S_{i}$ for $i \in \{ 1, 2 \}$ can be thought of as being 
 ordered and then placed to the right of the lowest left cell in 
 the tableau $T_{i}$ 
 and then having their labels shifted upwards by $1$. 
 The term $ a(1) \binom{3}{2}^{2} = 9$ in the expansion in \eqref{a4binomialsum}
 is in correspondence with the $9$ pairs $(T_{3}, T_{4})$ of order-$4$ standard immaculate tableaux 
 of the same shape with $2$ cells in the bottom row apart from 
 the $1$-labeled cell, as illustrated with 
 \begin{equation*}
 \Bigg( \raisebox{-0.8ex}{$ \begin{ytableau} 
 *(white) 1 
 \end{ytableau}, \ 
 \begin{ytableau} 
 *(white) 1 
 \end{ytableau}, $} \ \{ 2, 3 \}, \ \{ 1, 2 \} \Bigg) \mapsto 
 \Bigg( \, \raisebox{0.5ex}{$ \begin{ytableau} 
 *(white) 2 \\ 
 *(green) 1 & *(white) 3 & *(white) 4 
\end{ytableau}, \ 
 \begin{ytableau}
 *(white) 4 \\ 
 *(green) 1 & *(white) 2 & *(white) 3 
\end{ytableau} $} \, \Bigg). 
\end{equation*}
 The term $ a(2) \binom{3}{1}^2 = 18 $ in \eqref{a4binomialsum} is in correspondence with the $18$ pairs $(T_{3}, T_{4})$ 
 of order-$4$ standard immaculate tableaux of the same shape
 with $1$ cell in the bottow row apart from the $1$-labeled cell, and this is illustrated with 
 \begin{equation*}
 \Bigg( \raisebox{-0.8ex}{$ \begin{ytableau} 
 *(white) 1 & *(white) 2 
 \end{ytableau}, \ 
 \begin{ytableau} 
 *(white) 1 & *(white) 2 
 \end{ytableau}, $} \ \{ 3 \}, \ \{ 1 \} \Bigg) \mapsto 
 \Bigg( \, \raisebox{0.5ex}{$ \begin{ytableau} 
 *(white) 2 & *(white) 3 \\ 
 *(green) 1 & *(white) 4 
\end{ytableau}, \ 
 \begin{ytableau}
 *(white) 3 & *(white) 4 \\ 
 *(green) 1 & *(white) 2 
\end{ytableau} $} \, \Bigg) 
\end{equation*}
 and with 
 \begin{equation*}
 \Bigg( \raisebox{0.5ex}{$ \begin{ytableau} 
 *(white) 2 \\ 
 *(white) 1 
 \end{ytableau}, \ 
 \begin{ytableau} 
 *(white) 2 \\
 *(white) 1 
 \end{ytableau}, $} \ \{ 1 \}, \ \{ 2 \} \Bigg) \mapsto 
 \left( \, \raisebox{2.4ex}{$ \begin{ytableau} 
 *(white) 4 \\ 
 *(white) 3 \\ 
 *(green) 1 & *(white) 2 
\end{ytableau}, \ 
 \begin{ytableau}
 *(white) 4 \\ 
 *(white) 2 \\ 
 *(green) 1 & *(white) 3 
\end{ytableau} $} \, \right). 
\end{equation*}
 The $ a(3) \binom{3}{0}^2 = 7$ term in \eqref{a4binomialsum} 
 is in correspondence with the $7$ pairs 
 $(T_{3}, T_{4})$ of order-$4$ 
 standard immaculate tableaux of the same shape 
 with $0$ cells in the bottom row apart from the $1$-labeled cell, and this is illustrated with 
 \begin{equation*}
 \Bigg( \raisebox{-0.7ex}{$ \begin{ytableau} 
 *(white) 1 & *(white) 2 & *(white) 3 
 \end{ytableau}, \ 
 \begin{ytableau} 
 *(white) 1 & *(white) 2 & *(white) 3 
 \end{ytableau}, $} \ \varnothing, \ \varnothing \Bigg) \mapsto 
 \left( \, \raisebox{0.8ex}{$ \begin{ytableau} 
 *(white) 2 & *(white) 3 & *(white) 4 \\ 
 *(green) 1 
\end{ytableau}, \ 
 \begin{ytableau}
 *(white) 2 & *(white) 3 & *(white) 4 \\ 
 *(green) 1 
\end{ytableau} $} \, \right). 
\end{equation*}
 and with 
 \begin{equation*}
 \Bigg( \raisebox{0.55ex}{$ \begin{ytableau} 
 *(white) 2 & *(white) 3 \\ 
 *(white) 1 
 \end{ytableau}, \ 
 \begin{ytableau} 
 *(white) 2 & *(white) 3 \\ 
 *(white) 1 
 \end{ytableau}, $} \ \varnothing, \ \varnothing \Bigg) \mapsto 
 \left( \, \raisebox{2.4ex}{$ \begin{ytableau} 
 *(white) 3 & *(white) 4 \\ 
 *(white) 2 \\ 
 *(green) 1 
\end{ytableau}, \ 
 \begin{ytableau}
 *(white) 3 & *(white) 4 \\ 
 *(white) 2 \\ 
 *(green) 1 
\end{ytableau} $} \, \right). 
\end{equation*}
\end{example}

   To construct a family of combinatorial objects to enumerate   the $a$-sequence according to the recursion in \eqref{anrec}, we require 
  the preliminaries given as follows.   For a finite tuple $t$, we let $\ell(t)$ denote its length or the number of its entries,  and we write $t  
  = (t_{1}, t_{2}, \ldots, t_{\ell(t)})$.   A \emph{weak composition} is a finite tuple of nonnegative integers.  Correspondingly, we let a  
 \emph{weak partition} refer to a finite tuple of weakly decreasing nonnegative integers.  

\begin{definition} 
  Let $n$ be a positive integer. Let $\lambda$ be a strictly decreasing weak partition ending with $0$  and with a largest part $\leq n - 
  1$. An \emph{immacutation} of order $n$  is a tuple $t$ of length $2 \ell(\lambda)$ such that $t_{1}$  and $t_{2}$ are subsets of $\{ 1, 
 2, \ldots, n-1 \}$ of size $n - 1 - \lambda_{1}$,   and such that $t_{2i-1}$ and $t_{2i}$ are subsets of $\{ 1, 2, \ldots, \lambda_{i-1} - 1 \}$ 
 of size $\lambda_{i-1} - 1 - \lambda_{i}$, for $i \in \{ 2, 3, \ldots, \ell(\lambda) \}$. 
 We refer to $\lambda$ as the \emph{class} of $t$. 
\end{definition} 

\begin{example}
 Let $n = 4$ and let $\lambda$ denote the weak partition $(3, 2, 0)$. Then 
\begin{equation}\label{firstimmacutation}
 t = \big( \underbrace{\varnothing, \varnothing}_{\lambda_{1} = 3}, 
 \underbrace{\varnothing, \varnothing}_{\lambda_{2} = 2}, 
 \underbrace{ \{ 1 \}, \{ 1 \} }_{\lambda_{3} = 0} \big) 
\end{equation}
 is an immacutation of order $n$, where the groupings of entries suggested in 
 \eqref{firstimmacutation} are meant to illustrate the following. 
 The initial entries $t_{1} = \varnothing$ and $t_{2} = \varnothing$ 
 are subsets of $ \{ 1, 2, \ldots, n - 1 \} = \{ 1, 2, 3 \}$ of 
 size $n - 1 - \lambda_{1} = 0$, 
 and, letting $i = 2$, 
 the next two entries
 $t_{2i-1} = t_{3} = \varnothing$ and $t_{2i} = t_{4} = \varnothing$ 
 are subsets of $\{ 1, 2, \ldots, \lambda_{i-1} - 1 \} = \{ 1, 2 \}$ 
 of size $\lambda_{i-1} - 1 - \lambda_{i} = 0$, and, letting $i = 3$, 
 the next two entries $t_{2i-1} = t_{5} = \{ 1 \}$ and $t_{2i} = t_{6} = \{ 1 \}$
 are subsets of $\{ 1, 2, \ldots, \lambda_{i - 1} - 1 \} = \{ 1 \}$
 of size $\lambda_{i-1} - 1 - \lambda_{i} = 1$. 
\end{example}

\begin{theorem}\label{theoreman}
 For positive integers $n$, the number of immacutations of order $n$ is equal to $a(n)$. 
\end{theorem}

\begin{proof}
 Inductively, we assume that he number of order-$m$ immacutations is equal to $m$ for $m \leq n$, with the base case holding 
 in an immediate way. So, from the defining recurrence for the $a$-sequence in \eqref{anrec}, together with our inductive 
 hypothesis, it remains to prove that 
\begin{equation*}
 \text{$\#$ of order-$(n+1)$ immacutations} = 
 \sum_{k=0}^{n} \binom{n}{n-k}^{2} \big( \text{$\#$ of order-$k$ immacutations} \big). 
\end{equation*}
 Let $t$ denote an immacutation of order $k$ and class $\lambda$ By taking subsets $S_{1}$ and $S_{2}$ of $\{ 1, 2, \ldots, n \}$, 
 with each of $S_{1}$ and $S_{2}$ being of size $n - k$ for some fixed $k \in \{ 0, 1, \ldots, n \}$, we then form a tuple $u$ by 
 concatenating the $2$-tuple $(S_{1}, S_{2})$ and $t$. We then let $\mu$ denote the tuple obtained by concatenating the $1$-tuple 
 $(k)$ and $\lambda$. Since the order-$k$ immacutation $t$ is of class $\lambda$, the initial entry of $\lambda$ satisfies 
 $ \lambda_{1} \leq k - 1$, and any subsequent entries would form a strictly decreasing sequence ending with $0$, so we find that 
 $ \mu$ is a strictly decreasing weak partition ending with $0$ and with largest part $\leq n$. So, we find that $u$ is a tuple of length 
 $2 \ell(\mu)$ such that $u_{1} = S_{1}$ and $u_{2} = S_{2}$ are subsets of $\{ 1, 2, \ldots, n \}$ of size $n - \mu_{1} = n-k$. For $i = 
 2$, we have that $u_{2i-1} = t_{1}$ and $u_{2i} = t_{2}$, and, since $t$ is an immacutation of class $\lambda$ and order $k$, we 
 have that $t_{1}$ and $t_{2}$ are subsets of $\{ 1, 2, \ldots, k - 1 \}$ of size $ k - 1 - \lambda_{1}$, so that $u_{2i-1}$ and $u_{2 i}$ 
 (again for $i = 2$) are subsets of $\{ 1, 2, \ldots \mu_{i-1} - 1 \}$ of size $\mu_{i-1} - 1 - \mu_{2}$. The desired property for higher 
 indices $i$ then holds because $t$ is an immacutation. 
\end{proof}

\begin{example}
 The unique order-$4$ immacutation of class 
 $(3, 2, 1, 0)$ is $$(\varnothing, \varnothing, \varnothing, \varnothing, \varnothing, \varnothing, \varnothing, \varnothing ), $$ 
 and the order-$4$ immacutations of class $(3, 1, 0)$ are 
\begin{align*}
 & (\varnothing, \varnothing, \{ 1 \}, \{ 1 \}, \varnothing, \varnothing), \\ 
 & (\varnothing, \varnothing, \{ 1 \}, \{ 2 \}, \varnothing, \varnothing), \\ 
 & (\varnothing, \varnothing, \{ 2 \}, \{ 1 \}, \varnothing, \varnothing), \ \text{and} \\ 
 & (\varnothing, \varnothing, \{ 2 \}, \{ 2 \}, \varnothing, \varnothing), 
\end{align*}
 and we may confirm that there are $a(4) = 35$ order-$4$ immacutations in total. 
\end{example}

\begin{definition}\label{definitionf}
 Let $\mathcal{B}_{n}$ denote the set consisting of pairs of standard immaculate tableaux of the same shape $\alpha \vDash n$, 
 and let $\mathcal{C}_{n}$ denote the set consisting of immacutations of order $n$. Define 
\begin{equation}\label{displayfcolon}
 f\colon \mathcal{B}_{n} \to \mathcal{C}_{n} 
\end{equation}
 so that, for a pair $(T_{1}^{(1)}, T_{2}^{(1)})$ of standard immaculate tableaux of the same shape $\alpha \vDash n$, we take the 
 (possibly empty) set $S_{1}^{(1)}$ (resp.\ $S_{2}^{(1)}$) of labels to the right of the $1$-cell in the first row of $T_{1}^{(1)}$ (resp.\ 
 $T_{2}^{(1)}$) and we then form a set $\overline{S_{1}^{(1)}}$ (resp.\ $\overline{S_{2}^{(1)}}$) by shifting each member of 
 $S_{1}^{(1)}$ (resp.\ $S_{2}^{(1)}$) downwards by $1$, and we record the sets $\overline{S_{1}^{(1)}}$ and $\overline{S_{2}^{(1)}}$ 
 as consecutive entries to form an initial tuple $t^{(1)}$. We then recursively define $t^{(i)}$ and $T_{1}^{(i)}$ (resp.\ $T_{2}^{(i)}$) as 
 follows. We form $T_{1}^{(i)}$ (resp.\ $T_{2}^{(i)}$) by taking $T_{1}^{(i-1)}$ (resp.\ $T_{2}^{(i-1)}$) and removing its first row and 
 by then relabelling any remaining labels with consecutive integers starting with $1$, according to the increasing order of any labels in 
 $T_{1}^{(i-1)}$ (resp.\ $T_{2}^{(i-1)}$). By taking the (possibly empty) set $S_{1}^{(i)}$ (resp. $S_{2}^{(i)}$) of labels to the right of the 
 $1$-cell in the first row of $T_{1}^{(i)}$ (resp.\ $T_{2}^{(i)}$), we then form a set $\overline{S_{1}^{(i)}}$ (resp.\ $\overline{S_{2}^{(i)}}$) 
 by shifting each member of $S_{1}^{(i)}$ (resp.\ $S_{2}^{(i)}$) downwards by $1$, and we then set $t^{(i)}$ as the concatenation of 
 $t^{(i-1)}$ and the tuple $\big( \overline{S_{1}^{(i)}}, \overline{S_{2}^{(i)}} \big)$. We then let $f$ map $(T_{1}^{(1)}, T_{2}^{(1)})$ 
 to the immacutation $t^{(j)}$ obtained after the given procedure terminates (just before empty tableaux would be obtained from 
 successive truncations). 
\end{definition}

\begin{theorem}\label{theoremBC}
 The mapping in \eqref{displayfcolon} is a bijection from $\mathcal{B}_{n}$ to $ \mathcal{C}_{n}$. 
\end{theorem}

\begin{proof}
 This can be shown using the bijection in the proof of Theorem \ref{theoremconjectured} together with the bijective approach applied 
 to prove Theorem \ref{theoreman}, leaving the details to the reader. 
\end{proof}

\begin{example}
 Starting with the pair 
 \begin{equation}\label{startpairBtoC}
 \begin{ytableau} 
 *(white) 4 & *(white) 5 \\ 
 *(white) 3 \\ 
 *(white) 2 & *(white) 6 \\ 
 *(white) 1 
 \end{ytableau} \ \ \ \ \ \ \ \ 
 \begin{ytableau} 
 *(white) 5 & *(white) 6 \\ 
 *(white) 3 \\ 
 *(white) 2 & *(white) 4 \\ 
 *(white) 1 
 \end{ytableau} 
\end{equation}
 of standard immaculate tableaux of the same shape 
 $(1, 2, 1, 2) \vDash 6$, we have, in the notation of 
 Definition \ref{definitionf}, that 
 $$ t_{1} = \left( \varnothing, \varnothing \right), $$ 
 and the truncation-and-relabelling process in Definition \ref{definitionf} 
 applied to the pair in \eqref{startpairBtoC} gives rise to the pair 
 \begin{equation}\label{2202225202622282521224P22M2A}
 \begin{ytableau} 
 *(white) 3 & *(white) 4 \\ 
 *(white) 2 \\ 
 *(white) 1 & *(white) 5 
 \end{ytableau} \ \ \ \ \ \ \ \ 
 \begin{ytableau} 
 *(white) 4 & *(white) 5 \\ 
 *(white) 2 \\ 
 *(white) 1 & *(white) 3 
 \end{ytableau}, 
\end{equation}
 which, in turn, yields the tuple 
 $$ t_{2} = \left( \varnothing, \varnothing, \{ 4 \}, \{ 2 \} \right). $$ 
 An application of the truncation-and-relabelling procedure from 
 Definition \ref{definitionf} to 
 \eqref{2202225202622282521224P22M2A} yields the pair 
 \begin{equation*}
 \begin{ytableau} 
 *(white) 2 & *(white) 3 \\ 
 *(white) 1 
 \end{ytableau} \ \ \ \ \ \ \ \ 
 \begin{ytableau} 
 *(white) 2 & *(white) 3 \\ 
 *(white) 1 
 \end{ytableau}, 
\end{equation*}
 giving rise to the tuple 
 $$ t_{3} = \left( \varnothing, \varnothing, \{ 4 \}, \{ 2 \}, 
 \varnothing, \varnothing \right). $$ 
 Mimicking the above steps, we obtain the pair 
 \begin{equation*}
 \begin{ytableau} 
 *(white) 1 & *(white) 2 
 \end{ytableau} \ \ \ \ \ \ \ \ 
 \begin{ytableau} 
 *(white) 1 & *(white) 2 
 \end{ytableau}, 
\end{equation*}
 and this gives us the tuple 
\begin{equation}\label{t4immacutation}
 t_{4} = \left( \varnothing, \varnothing, \{ 4 \}, \{ 2 \}, 
 \varnothing, \varnothing, \{ 1 \}, \{ 1 \} \right). 
\end{equation}
 Observe that the tuple in \eqref{t4immacutation}
 is an immacutation of order $n = 6$ and of class 
 $\lambda = (5, 3, 2, 0)$. 
 In this direction, we find that 
 $t_{1}$ and $t_{2}$ are subsets of $\{ 1, 2, \ldots, n - 1 \} = \{ 1, 2, 3, 4, 5 \}$ 
 of size $n - 1 - \lambda_{1} = 0$, 
 and, setting $i = 2$, the following entries
 $t_{2i-1} = \{ 4 \}$ and $t_{2i} = \{ 2 \}$ 
 are subsets of $\{ 1, 2, \ldots, \lambda_{i-1} - 1 \} = \{ 1, 2, 3, 4 \}$ 
 of size $\lambda_{i-1} - 1 - \lambda_{i} = 1$, 
 and so forth. 
\end{example}

 By Theorem \ref{theoremBC}, by taking the dimension of both sides of \eqref{CIndefinition}, we obtain a proof of the 
 Frobenius--Young-type identity in \eqref{immacutationsFY}. By letting $\{ e_{T_{1}, T_{2}}^{\alpha} \}$ denote a matrix basis of 
 $ \mathbb{C}\mathcal{I}_{n}$ indexed by pairs $(T_{1}, T_{2})$ of standard immaculate tableaux of the same
 shape $\alpha \vDash n$, 
 Theorem \ref{theoremBC} provides us with an equivalent basis 
 that we denote as $\{ e_{f(T_{1}, T_{2})}^{\alpha} \}$ and that is indexed by order-$n$ immacutations. 

\subsection{The structure of immaculate algebras}
 Since we have focused on the combinatorics underlying the evaluation of $\text{dim} \, \mathbb{C}\mathcal{I}_{n}$
 and the dimensions of the components in the semisimple decomposition of $\mathbb{C}\mathcal{I}_{n}$, 
 this leads us to turn toward considering the structure of $\mathbb{C}\mathcal{I}_{n}$, as below. 

\begin{theorem}\label{theoremcontains}
 For all $n$, the immaculate algebra $\mathbb{C}\mathcal{I}_{n}$
 contains an isomorphic copy of $\mathbb{C}S_{n}$. 
\end{theorem}

\begin{proof}
 For an arbitrary integer partition $\lambda \vdash n$, we see that $f^{\lambda} \leq g^{\lambda}$, since every standard Young 
 tableau of shape $\lambda$ is a standard immaculate tableau of the same shape. So, since 
 $\mathcal{M}_{f^{\lambda}}(\mathbb{C})$ is a subalgebra of $\mathcal{M}_{g^{\lambda}}(\mathbb{C})$, 
 and similarly for direct sums of matrix algebras of the form $\mathcal{M}_{f^{\lambda}}(\mathbb{C})$, 
 we have that 
 $\bigoplus_{\lambda \vdash n} \mathcal{M}_{f^{\lambda}}(\mathbb{C})$ is a subalgebra 
 of $ \bigoplus_{\alpha \vDash n} \mathcal{M}_{g^{\alpha}}(\mathbb{C})$, so that the desired 
 result then follows from the consequence of Young's construction in \eqref{decomposeCSn}. 
\end{proof}

 The connection between immacutations and permutations can be made more explicit
 in the following sense. 
 If we set $b(0) = 1$ and
 $$ b(n) = \sum_{k=0}^{n-1} \frac{ (n-1)! }{k!} b(k) $$
 by direct analogy with \eqref{anrec}, 
 then $b(n) = n!$ for all $n$. 
 This and Theorem \ref{theoremcontains} 
 point toward why 
 it would be desirable to construct a basis $B$ of 
 $\mathbb{C}\mathcal{I}_{n}$ indexed by order-$n$ immacutations 
 in such a way so that the product of two such basis elements 
 is nonzero and a member of the same basis $B$ (or a scalar multiple of an element in $B$), 
 according to a combinatorial rule by analogy with the composition of permutations. 
 This leads us to make use of an analogue of the following result 
 that is due to Hewitt and Zuckermann \cite{HewittZuckerman1955} and that is reproduced 
 in Clifford and Preston's text on the algebraic theory of semigroups \cite[p.\ 167]{CliffordPreston1961}. 
 Recall that a \emph{semigroup} is a set endowed with associative binary operation, 
 and a \emph{monoid} is a semigroup with an identity element. 

\begin{theorem}
 (Hewitt $\&$ Zuckerman) For a semisimple algebra $A$ over a field $F$, 
 if $A \cong F S$ for a finite semigroup $S$, then one of the simple 
 components of $A$ is of order $1$ over $F$, and the converse holds if $F$ is algebraically closed \cite{HewittZuckerman1955}. 
\end{theorem}

 Observe that $g^{\alpha} = 1$ in the cases whereby the composition $\alpha$
 has a vertical or horizontal shape. 
 So, by the Hewitt--Zuckerman Theorem, the semisimple algebra $\mathbb{C}\mathcal{I}_{n}$ 
 has a semigroup algebra structure. 
 To construct an analogue of the composition of permutations, we would want to 
 construct a basis of $\mathbb{C}\mathcal{I}_{n}$ indexed by immacutations 
 in such a way so as to give $\mathbb{C}\mathcal{I}_{n}$ the structure of a monoid algebra. 

\begin{theorem}\label{theoremmonoid}
 For all positive integers $n$, the immaculate algebra $\mathbb{C}\mathcal{I}_{n}$
 has the structure of a monoid algebra. 
\end{theorem}

\begin{proof}
 The $1$-dimensional case for $n = 1$ holds in an immediate way, so we let $n > 2$. 
 In this case, we have that there are at least two distinct simple 
 components that are of $\mathbb{C}\mathcal{I}_{n}$ and that are of order $1$ over $\mathbb{C}$, 
 and these components correspond to the integer compositions $\alpha \vDash n$
 of vertical and horizontal shapes. 
 Letting $\{ e_{T_{1}, T_{2}}^{\alpha} \}$ denote the same matrix unit basis 
 for $\mathbb{C}\mathcal{I}_{n}$ constructed above, we then define 
 \[ h_{T_{1}, T_{2}}^{\alpha} = \begin{cases} 
 e_{T_{1}, T_{2}}^{(1^n)} & \text{if $\alpha = (1^n)$,} \\
 \sum_{\alpha \vDash n} e_{T, T}^{\alpha} & \text{if $\alpha = (n)$,} \\
 e_{T_{1}, T_{2}}^{(1^n)} + e_{T_{1}, T_{2}}^{\alpha} & \text{otherwise.} 
 \end{cases}
\]
  We find that the family $\{ h_{T_{1}, T_{2}}^{\alpha} \}$ is a basis of  $\mathbb{C} \mathcal{I}_{n}$, by considering the transition matrix 
 given by expanding $h$-expressions into the $e$-basis,  letting compositions be ordered so that  $(1^{n}) < (n)$, yielding a (square) 
 transition matrix with ones along the main diagonal, ones along the first column,  certain $1$-entries along the second row, and zeroes 
 everywhere else,  and this is necessarily of determinant $1$. Since 
\begin{equation}\label{idhorizontal}
 \Yvcentermath1 \Yboxdim{16pt} 
 h_{\young(12{\cdots}n),\young(12{\cdots}n)}^{(n)} 
 = \sum_{\alpha \vDash n} e_{T, T}^{\alpha}
 \end{equation} 
   is the sum of all idempotent matrix units in a matrix unit basis  for $\mathbb{C}\mathcal{I}_{n}$, we write $\text{id}$  in place of the  
  basis element on the left-hand side of    \eqref{idhorizontal}.  We also write $\text{qid}$ (in reference to a ``quasi-identity'' element) in  
 place of the $h$-element  
\begin{equation*}
  \Yvcentermath1 \Yboxdim{16pt} 
 h_{\young(n,{\vdots},2,1),\young(n,{\vdots},2,1)}^{(1^n)} 
 = e_{\young(n,{\vdots},2,1),\young(n,{\vdots},2,1)}^{(1^n)}. 
 \end{equation*} 
 We find that 
 $$ h_{T_{1}, T_{2}}^{\alpha} h_{T_{3}, T_{4}}^{\beta} = h_{T_{1}, T_{2}}^{\alpha} $$
 if $\beta = (n)$ and 
 $$ h_{T_{1}, T_{2}}^{\alpha} h_{T_{3}, T_{4}}^{\beta} = h_{T_{3}, T_{4}}^{\beta} $$
 and $\alpha = (n)$. 
 If $\alpha \neq (n)$ and $\beta \neq (n)$, then 
 \[ h_{T_{1}, T_{2}}^{\alpha} h_{T_{3}, T_{4}}^{\beta} = \begin{cases} 
 \text{qid} & \text{if $T_{2} \neq T_{3}$ or $\alpha \neq \beta$,} \\ 
 h_{T_{1}, T_{4}}^{\alpha} & \text{otherwise.} 
 \end{cases}
\]
 This gives us that $\{ h_{T_{1}, T_{2}}^{\alpha} \}$ is closed under the multiplicative operation
 on $\mathbb{C}\mathcal{I}_{n}$, giving $\{ h_{T_{1}, T_{2}}^{\alpha} \}$ the structure 
 of a monoid under the product operation on $\mathbb{C}\mathcal{I}_{n}$, 
 with $\mathbb{C}\mathcal{I}_{n} $ equal to the monoid algebra
 $\mathbb{C} \{ h_{T_{1}, T_{2}}^{\alpha} \}$. 
\end{proof}

 Let integer compositions of $n$ be linearly ordered so that $(1^{n}) \triangleleft (n)$ and so that $(n) \triangleleft \alpha$ for 
 $\alpha \not\in \{ (1^n), (n) \}$ and so that $\alpha \triangleleft \beta$ for $\alpha, \beta \not\in \{ (1^n), (n) \}$ if and only if 
 $\alpha <_{\ell} \beta$ for the lexicographic ordering $<_{\ell}$ on compositions. The then sort the set of order-$n$ immacutations 
 so that, for two such immacutations $t_1$ and $t_2$, we have that $t_1 < t_{2}$ if and only if the shape of either tableau in the pair 
 $f^{-1}(t_1)$ is strictly less than (with respect to $\triangleleft$) the shape of either tableau in the pair $f^{-1}(t_2)$ or $t_1$ is strictly 
 less than $t_2$ lexicographically, by identifying $\varnothing$ with $0$ and by identifying a set of positive integers with the word 
 obtained by sorting its entries. Disregarding superscripts for $h$-basis elements, we then define the product of $t_1$ and $t_2$ as 
 the inverse image under $f$ of the index of $h_{f^{-1}(t_1)} h_{f^{-1}(t_2)}$. 

\begin{example}
 There are $7$ immacutations of order $3$, which are ordered according to $\triangleleft$ so that 
\begin{align*}
 t_{1} & = (\varnothing, \varnothing, \varnothing, \varnothing, \varnothing, \varnothing), \\ 
 t_{2} & = ( \{ 1, 2 \}, \{ 1, 2 \} ), \\
 t_{3} & = ( \varnothing, \varnothing, \{ 1 \}, \{ 1 \} ), \\ 
 t_{4} & = ( \{ 1 \}, \{ 1 \}, \varnothing, \varnothing ), \\ 
 t_{5} & = ( \{ 1 \}, \{ 2 \}, \varnothing, \varnothing ), \\ 
 t_{6} & = ( \{ 2 \}, \{ 1 \}, \varnothing, \varnothing ), \\
 t_{7} & = ( \{ 2 \}, \{ 2 \}, \varnothing, \varnothing ). 
\end{align*}
 We thus have that $\{ t_1, t_2, \ldots, t_7 \}$ has the structure of a monoid 
 under the operation of $\mathbb{C}\mathcal{I}_{3}$ in the manner described above. Explicitly, 
 a composition table for the monoid $\{ t_1, t_2, \ldots, t_7 \}$ is given below. 
\begin{center}
 \begin{tabular}{ c | c c c c c c c }
 \null & $t_1$ & $t_2$ & $t_3$ & $t_4$ & $t_5$ & $t_6$ & $t_7$ \\ \hline
 $t_1$ & $t_1$ & $t_1$ & $t_1$ & $t_1$ & $t_1$ & $t_1$ & $t_1$ \\ 
 $t_2$ & $t_1$ & $t_2$ & $t_3$ & $t_4$ & $t_5$ & $t_6$ & $t_7$ \\
 $t_3$ & $t_1$ & $t_3$ & $t_3$ & $t_1$ & $t_1$ & $t_1$ & $t_1$ \\
 $t_4$ & $t_1$ & $t_4$ & $t_1$ & $t_4$ & $t_5$ & $t_1$ & $t_1$ \\
 $t_5$ & $t_1$ & $t_5$ & $t_1$ & $t_1$ & $t_1$ & $t_4$ & $t_5$ \\
 $t_6$ & $t_1$ & $t_6$ & $t_1$ & $t_6$ & $t_7$ & $t_1$ & $t_1$ \\
 $t_7$ & $t_1$ & $t_7$ & $t_1$ & $t_1$ & $t_1$ & $t_6$ & $t_7$ 
 \end{tabular}
\end{center}
\end{example}

\section{Conclusion}
 The immaculate algebra $\mathbb{C}\mathcal{I}_{n}$ 
 may be considered in relation to the partition algebra $\mathbb{C}_n(x)$, adapting notation from
 Halverson and Ram \cite{HalversonRam2005}, 
 since, for a suitable parameter $x$, we have that $\mathbb{C}_n(x)$ and 
 $\mathbb{C}\mathcal{I}_{n}$ are both semisimple and 
 both contian $\mathbb{C}S_{n}$ and both contain two inequivalent $1$-dimensional 
 submodules that correspond in natural ways 
 to the $1$-dimensional Sphect modules associated with horizontal and vertical shapes. 
 The partition algebra $\mathbb{C}A_n(x)$ 
 has been considered as a \emph{twisted semigroup algebra} by Wilcox 
 \cite{Wilcox2007}, but it seems that 
 $\mathbb{C}A_n(x)$ has not previously been considered as a semigroup algebra or 
 as a monoid algebra. Using Halverson and Ram's matrix unit construction \cite{HalversonRam2005}, 
 and adapting our proof of Theorem \ref{theoremmonoid} 
 this would give $\mathbb{C}A_n(x)$ the structure of a monoid algebra, 
 and we encourage the exploration of combinatorial properties associated with this monoid algebra. 

\subsection*{Acknowledgements}
 The author thanks Mike Zabrocki and Laura Colmenarejo for a useful discussion at the Fields Institute, 
 and thanks Karl Dilcher for useful discussions. 
 The author was supported by a Killam Postdoctoral Fellowship from the Killam Trusts. 



\end{document}